\RequirePackage{amsmath}
\documentclass[runningheads]{llncs}
\usepackage[T1]{fontenc}
\usepackage{hyperref}       
\usepackage{url}            
\usepackage{booktabs}       
\usepackage{amsfonts}       
\usepackage{tablefootnote}
\usepackage{verbatim}
\usepackage{nicefrac}       
\usepackage{microtype}      
\usepackage{lipsum}
\usepackage{mathtools}

\usepackage{amsmath,amssymb}
\usepackage{algorithm,algorithmic}
\usepackage{pifont}
\usepackage{cases}
\usepackage{subcaption,graphicx}
\usepackage{stackengine}    
\usepackage{wrapfig}
\usepackage{enumitem}

\newtheorem{assumption}[theorem]{Assumption}

\newcommand{\eps}{\varepsilon}

\newcommand{\ol}{\overline}
\newcommand{\one}{\mathbf{1}}

\newcommand{\circledOne}{\text{\ding{172}}}
\newcommand{\circledTwo}{\text{\ding{173}}}
\newcommand{\circledThree}{\text{\ding{174}}}

\renewcommand{\hat}{\widehat}

\newcommand{\numberthis}{\addtocounter{equation}{1}\tag{\theequation}}

\DeclareMathOperator*{\argmin}{arg\,min}

\DeclareMathOperator{\prox}{prox}

\DeclareMathOperator{\col}{col}

\newcommand{\R}{\mathbb{R}}

\newcommand{\mI}{{\bf I}}

\newcommand{\mP}{{\bf P}}

\newcommand{\mW}{{\bf W}}

\newcommand{\cE}{{\mathcal{E}}}

\newcommand{\cG}{{\mathcal{G}}}

\newcommand{\cL}{{\mathcal{L}}}

\newcommand{\cV}{{\mathcal{V}}}

\newcommand{\bu}{{\bf u}}
\newcommand{\bv}{{\bf v}}

\newcommand{\bx}{{\bf x}}
\newcommand{\by}{{\bf y}}
\newcommand{\bz}{{\bf z}}

\newcommand{\muav}{\mu_{g}}
\newcommand{\mumin}{\mu_{l}}
\newcommand{\Lav}{L_{g}}
\newcommand{\Lmax}{L_{l}}

\newcommand{\ds}{\displaystyle}
\newcommand{\norm}[1]{\left\| #1 \right\|}

\newcommand{\angles}[1]{\left\langle #1 \right\rangle}
\newcommand{\cbraces}[1]{\left( #1 \right)}
\newcommand{\sbraces}[1]{\left[ #1 \right]}

\def\<#1,#2>{\langle #1,#2\rangle}

\usepackage{color}

\begin{document}

\title{Decentralized Proximal Optimization Method with Consensus Procedure\thanks{The research is supported by the Ministry of Science and Higher Education of the Russian Federation (Goszadaniye) 075-00337-20-03, project No. 0714-2020-0005.}}

\titlerunning{Decentralized Proximal Method}

\author{Alexander Rogozin\inst{1}\orcidID{0000-0003-3435-2680} \and
Anton Novitskii\inst{1}\orcidID{0009-0005-6411-2639} \and
Alexander Gasnikov\inst{1,2,3}\orcidID{0000-0002-7386-039X}}

\authorrunning{A. Rogozin et al.}

\institute{
	Moscow Institute of Physics and Technology, Moscow, Russia \and
	Institute for Information Transportation Problems, Moscow, Russia \and
	Caucasus Mathematic Center of Adygh State University, Moscow, Russia
}

\maketitle              

\begin{abstract}
Decentralized optimization is well studied for smooth unconstrained problems. However, constrained problems or problems with composite terms are an open direction for research. We study structured (or composite) optimization problems, where the functional is a sum of a convex smooth function and a proper convex proximal-friendly term. Our method builds upon an accelerated proximal gradient descent and makes several consensus iterations between computations.

\end{abstract}

\section{Introduction}

Distributed optimization has a wide range of applications. The fields where distributed optimization problems arise include power system control \cite{ram2009distributed,gan2012optimal}, formation control \cite{olshevsky2010efficient,ren2006consensus,jadbabaie2003coordination}, distributed statistical inference and machine learning \cite{rabbat2004distributed,forero2010consensus,nedic2017fast,nedic2020distributed} distributed coordination and control \cite{ren2008distributed}, distributed averaging \cite{cai2014average,olshevsky2014linear,xiao2007distributed}, distributed spectrum sensing \cite{bazerque2009distributed}. Distributed optimization takes place when the data is separated between several computational entities due to a large amount of datasets, privacy constraints or a split structure of the data itself.

In this paper we focus on decentralized systems. Several computational nodes, or agents, locally hold objective functions and can communicate to each other. A centralized aggregator is not present in the network, so the agents directly communicate to each other. The structure of the network is represented as an undirected graph, and the agents exchange information with their immediate neighbors.

The communication network may also change with time, which typically happens due to technical instabilities \cite{rogozin2022decentralized_survey}. A time-varying network corresponds to a changing communication graph.

Informally speaking, the complexity of a decentralized method depends on objectives condition number $\kappa$ and on graph condition number $\chi$. For static networks, optimal methods require $O(\sqrt\kappa\sqrt\chi\log(1/\eps))$ communication rounds to reach $\eps$-accuracy. For optimization over time-varying networks, optimal algorithms require $O(\kappa\sqrt\chi\log(1/\eps))$ communications.

\noindent\textbf{Related work}. Previously in the literature the classes or time-varying and time-static graphs have been studied. For functions with smooth gradients, lower communication and local computation complexity bounds were proposed in \cite{scaman2017optimal}. Optimal dual \cite{scaman2017optimal} and primal \cite{kovalev2020optimal} algorithms were developed, as well. Lower bounds for problems with non-smooth Lipschitz objectives were proposed in \cite{scaman2018optimal} along with methods optimal up to a factor dependent on space dimension. Paper \cite{dvinskikh2019decentralized} proposed an optimal primal scheme based on gradient sliding. A method for problems with composite terms was developed in \cite{ye2020multi}.

There is a group of algorithms that use a consensus subroutine technique. Initially this approach has been proposed in \cite{jakovetic2014fast} for time-static graphs. After that, it was applied to time-varying graphs. Deterministic setup was studied in \cite{rogozin2021towards} and stochastic setup with mini-batching was studied in \cite{rogozin2021accelerated}. The technique has also been applied to saddle-point problems \cite{beznosikov2021near,beznosikov2021distributed_2}. Paper \cite{beznosikov2021distributed_2} supports composite min-max problems, but the analysis requires a bounded constraint set. Our assumptions do not require bounded gradient norm or bounded constraint set.

\noindent\textbf{Our contribution}. We continue a series of works on consensus subroutine, proposing a decentralized proximal method. Our analysis is relatively easy and we do not require bounded gradients or bounded constraint set. Our approach only requires Lipschitz gradient and strong convexity.

\section{Problem Statement}\label{sec:problem_statement}

In this paper, we study a sum-type minimization problem
\begin{align}\label{eq:initial_problem}
	\min_{x\in Q}~ f(x) + g(x) = \frac{1}{m}\sum_{i=1}^m f_i(x) + g(x)
\end{align}
Here functions $f_i(x)$ are convex and smooth, and function $g(x)$ is a proper convex closed function, possibly non-smooth, and $Q$ is a closed convex set. We assume that $g(x)$ is proximal-friendly, i.e. its proximal operator can be easily computed.

\subsection{Notation}

Let $\otimes$ denote the Kronecker product. Let $\bx = \col[x_1, \ldots, x_m] = [x_1^\top, \ldots, x_m^\top]^\top\in\R^{md}$ denote a column vector. Let $\lambda_{\max}(\cdot)$ and $\lambda_{\min}^+(\cdot)$ denote maximum and minimum nonzero eigenvalues of a matrix. Also let $\cL = \{x_1 = \ldots = x_m\}$ denote the consensus constraint set. We denote $\one$ a vector of all ones and introduce a projection operator $\mP = (1/m)\one\one^\top\otimes\mI$ (the dimension is known from the context).

We introduce $\bx = \col[x_1, \ldots, x_m]$ and denote
\begin{align}\label{eq:def_F_G}
	F(\bx) = \sum_{i=1}^m f_i(x_i),~ G(\bx) = \sum_{i=1}^m g_i(x_i).
\end{align}
We also introduce a prox-operator for $g$ w.r.t. set $Q$:
\begin{align*}
	\prox_g^\gamma(x) = \argmin_{y\in Q}\cbraces{g(y) + \frac{1}{2\gamma}\norm{y - x}_2^2}.
\end{align*}
Analogously, a prox-operator for $G$ w.r.t. $Q^m = \{\bx\in\R^{md}:~ x_i\in Q,~ i = 1, \ldots, m\}$ writes as 
\begin{align*}
	\prox_G^\gamma(\bx) = \argmin_{\by\in Q^m}\cbraces{G(\bx) + \frac{1}{2\gamma}\norm{\by - \bx}_2^2} = \col[\prox_g^\gamma(x_1)\ldots \prox_g^\gamma(x_m)].
\end{align*}
Note that $\prox_G^\gamma(\bx)$ is a separable operator, i.e. it can be computed separately for $x_1, \ldots, x_m$.

\subsection{Objective Functions}

Our paper focuses on smooth strongly convex functions.
\begin{assumption}\label{assum:smoothness}
	For each $i = 1, \ldots, m$ function $f_i$ is $L_i$-smooth, i.e. for any $x, y\in\R^d$ it holds
	\begin{align*}
		f_i(y)\leq f_i(x) + \angles{\nabla f_i(x), y - x} + \frac{L_i}{2}\norm{y - x}_2^2.
	\end{align*}
\end{assumption}
\begin{assumption}\label{assum:strong_convexity}
	For each $i = 1, \ldots, m$ function $f_i$ is $\mu_i$-strongly convex, i.e. for any $x, y\in\R^d$ it holds
	\begin{align*}
		f_i(y)\geq f_i(x) + \angles{\nabla f_i(x), y - x} + \frac{\mu_i}{2}\norm{y - x}_2^2.
	\end{align*}
\end{assumption}
We rewrite problem \eqref{eq:initial_problem} as
\begin{align}\label{eq:problem_linear_constraints}
	\min_{\bx\in\R^{md}}~ &F(\bx) + G(\bx) \\
	\text{s.t. } &x_1 = \ldots = x_m \nonumber
\end{align}
Also introduce local and global constants characterizing problem optimization parameters.
\begin{subequations}\label{eq:local_global constants}
	\begin{align}
		\Lmax &= \max_{i = 1, \ldots, m} L_i,~ \mumin = \min_{i = 1, \ldots, m} \mu_i, \\
		\Lav &= \frac{1}{m}\sum_{i=1}^m L_i,~ \muav = \frac{1}{m}\sum_{i=1}^m \mu_i.
	\end{align}
\end{subequations}
It is known that global and local constants may significantly differ \cite{scaman2017optimal,rogozin2022decentralized_survey}.

\subsection{Communication Network}

\begin{assumption}
	
	We assume that nodes are connected via a \textit{time-varying} network represented by a sequence of graphs $\{\cG^k = (\cV, \cE^k)\}_{k=0}^\infty$. The graphs have a common set of vertices $\cV$ but may have different edge sets $\cE^k$. With each of the graphs, we associate a mixing matrix $W^k$.
	
	Mixing matrix sequence $\{W^k\}_{k=1}^\infty$ satisfies the following properties:
	\begin{enumerate}
		\item (Network compatibility) For each $k = 1, 2, \ldots$ it holds $[W^k]_{ij} = 0$ if $(i, j)\notin\cE^k$.
		\item (Double stochasticity) For each $k = 1, 2, \ldots$ it holds $W^k\one = \one,~ \one^\top W^k = \one^\top$.
		\item (Spectral property) There exists $\lambda < 1$ such that for all $k = 1, 2, \ldots$ it holds $\norm{W^k - \dfrac{1}{m}\one\one^\top}_2\leq 1 - \chi^{-1}$.
	\end{enumerate}
\end{assumption}
We also introduce $\mW^k = W^k\otimes\mI$.

\section{Inexact Oracle Framework}\label{sec2}

Let us construct an inexact model for function $h(x) = \frac{1}{m}\sum_{i=1}^m f_i(x) + g(x)$.

\begin{lemma}\label{lem:inexact_oracle}
	Consider $y\in\R^d,~ z\in\R^d,~ \bx = \col[x_1, \ldots, x_m]\in\R^{md}$. Define
	\begin{align*}
		\eta &= \frac{1}{2m}\cbraces{\frac{\Lmax^2}{\Lav} + \frac{2\Lmax^2}{\muav} + \Lmax - \mumin}, \numberthis\label{eq:eta_inexact_oracle} \\
		\delta &= \eta \sum_{i=1}^m\norm{x_i - y}_2^2, \numberthis\label{eq:delta_inexact_oracle} \\
		f_{\delta}(y, \bx) &= \frac{1}{m} \sum_{i=1}^m \sbraces{f_i(x_i) + \angles{\nabla f_i(x_i), y - x_i} + \frac{1}{2}\cbraces{\mumin - \frac{2\Lmax^2}{\muav}}\norm{y - x_i}^2}, \\
		\psi_{\delta}(z, y, \bx) &= \frac{1}{m}\sum_{i=1}^m \sbraces{\angles{\nabla f_i(x_i), z - y} + g(z) - g_i(x_i)}.
	\end{align*}
	Then $(f_{\delta}(y, \bx), \psi_{\delta}(z, y, \bx))$ is a $(\delta, 2\Lav, \muav/2)$-model of $f$ at point $\ol x$, i.e.
	\begin{align*}
		\frac{\muav}{4}\norm{z - y}^2 \le f(z) - f_{\delta}(y, \bx) - \psi_\delta(z, y, \bx) \leq \Lav\norm{z - y}^2 + \delta.
	\end{align*}
\end{lemma}

\begin{proof}
	Denote $\by = \one\otimes y,~ \bz = \one\otimes z$. It is convenient to use the notation of $F(\bx), G(\bx)$. First, write the following lower bound.
	\begin{align*}
		F(\bz) + G(\bz)
		&\overset{\circledOne}{\geq} F(\bx) + \sbraces{\angles{\nabla F(\bx), \by - \bx} + \frac{\mumin}{2}\norm{\bx - \by}^2} \\
		&\qquad+ \sbraces{\angles{\mP\nabla F(\by), \bz - \by} + \frac{\muav}{2}\norm{\bz - \by}^2} + G(\bz) \\
		&= \sbraces{F(\bx) + \angles{\nabla F(\bx), \by - \bx} + \frac{\mumin}{2}\norm{\bx - \by}^2} + \angles{\mP\nabla F(\bx), \bz - \by} \\
		&\qquad + \angles{\mP(\nabla F(\by) - \nabla F(\bx)), \bz - \by} + \frac{\muav}{2}\norm{\by - \bz}^2 + G(\bz) \\
		&\overset{\circledTwo}{\geq} \sbraces{F(\bx) + \angles{\nabla F(\bx), \by - \bx} + \frac{\mumin}{2}\norm{\bx - \by}_2^2} + \angles{\mP\nabla F(\bx), \bz - \by} \\
		&\qquad- \frac{1}{\muav}\norm{\mP(\nabla F(\by) - \nabla F(\bx))}_2^2 - \frac{\muav}{4}\norm{\bz - \by}_2^2 + \frac{\muav}{2}\norm{\bz - \by}_2^2 + G(\bz) \\
		&\overset{\circledThree}{\geq} \sbraces{F(\bx) + G(\bx) + \angles{\nabla F(\bx), \by - \bx} + \cbraces{\frac{\mumin}{2} - \frac{\Lmax^2}{\muav}}\norm{\bx - \by}_2^2} \\
		&\qquad+ \sbraces{\angles{\mP\nabla F(\bx), \bz - \by} + G(\bz) - G(\bx)} + \frac{\muav}{4}\norm{\bz - \by}_2^2,
		\numberthis\label{eq:inexact_oracle_lower_bound_1}
	\end{align*}
	where $\circledOne$ holds since $F$ is $\mumin$-strongly convex over $R^{md}$ and $\muav$-strongly convex over $\cL$, $\circledTwo$ holds by Young inequality and $\circledThree$ holds by $\Lmax$-smoothness of $F$.
	
	Second, we get an upper estimate on $F(\bz) + G(\bz)$.
	\begin{align*}
		F(\bz) + G(\bz)
		&\overset{\circledOne}{\leq} F(\bx) + \sbraces{\angles{\nabla F(\bx), \by - \bx} + \frac{\Lmax}{2}\norm{\by - \bx}_2^2} \\
		&\qquad + \sbraces{\angles{\mP\nabla F(\by), \bz - \by} + \frac{\Lav}{2}\norm{\bz - \by}_2^2} + G(\bz) \\
		&=\sbraces{F(\bx) + \angles{\nabla F(\bx), \by - \bx} + \frac{\Lmax}{2}\norm{\by - \bx}_2^2} \\
		&\qquad+ \angles{\mP\nabla F(\bx), \bz - \by} + \frac{\Lav}{2}\norm{\bz - \by}_2^2 \\
		&\qquad+ \angles{\mP(\nabla F(\by) - \nabla F(\bx)), \bz - \by} + G(\bz) \\
		&\overset{\circledTwo}{\leq}\sbraces{F(\bx) + \angles{\nabla F(\bx), \by - \bx} + \frac{\Lmax}{2}\norm{\by- \bx}_2^2} \\
		&\qquad+ \angles{\mP\nabla F(\bx), \bz - \by} + \frac{\Lav}{2}\norm{\bz - \by}_2^2 \\
		&\qquad+ \frac{1}{2\Lav}\norm{\mP(\nabla F(\by) - \nabla F(\bx))}_2^2 + \frac{\Lav}{2}\norm{\bz - \by}_2^2 + G(\bz) \\
		&\overset{\circledThree}{\leq} \sbraces{F(\bx) + G(\bx) + \angles{\nabla F(\bx), \by - \bx} + \cbraces{\frac{\mumin}{2} - \frac{\Lmax^2}{\muav}}\norm{\by - \bx}_2^2} \\
		&\qquad+ \sbraces{\angles{\mP\nabla F(\bx), \bz - \by} + G(\bz) - G(\bx)} + \Lav\norm{\bz - \by}_2^2 \\
		&\qquad + \cbraces{\frac{\Lmax^2}{2\Lav} + \frac{\Lmax^2}{\muav} - \frac{\mumin}{2} + \frac{\Lmax}{2}}\norm{\by - \bx}_2^2,
		\numberthis\label{eq:inexact_oracle_lower_bound_2}
	\end{align*}
	where $\circledOne$ holds since $F$ is $\Lmax$-smooth over $\R^{md}$ and $\Lav$-smooth over $\cL$, $\circledTwo$ holds by Young inequality and $\circledThree$ holds by $\Lmax$-smoothness of $F$.
	
	It remains to recall the definitions of $F(\bx), G(\bx)$ from \eqref{eq:def_F_G} and combine \eqref{eq:inexact_oracle_lower_bound_1} and \eqref{eq:inexact_oracle_lower_bound_2} to get the desired inequality.
\end{proof}

\section{Accelerated Algorithm and Convergence}

\begin{algorithm}[H]
	\caption{Accelerated decentralized proximal method with consensus subroutine}
	\label{alg:decentralized_agd_prox}
	\begin{algorithmic}[1]
		\REQUIRE{Initial guess $\bx^0\in \cL$, constants $L, \mu > 0$, $\bu^0 = \bx^0$, $\alpha^0 = A^0 = 0$}
		\FOR{$k = 0, 1, 2,\ldots$}
		\STATE{Find $\alpha^{k+1}$ as the greater root of \\$(A^k + \alpha^{k+1})(1 + A^k \muav/2) = 2\Lav(\alpha^{k+1})^2$}
		\STATE{$A^{k+1} = A^k + \alpha^{k+1}$}
		\vspace{0.1cm}
		\STATE{$\ds \by^{k+1} = \frac{\alpha^{k+1} \bu^k + A^k \bx^k}{A^{k+1}}$}
		\vspace{0.1cm}
		\STATE{\label{alg_step:agd_step_prox}$\ds \bv^{k+1} = \frac{\alpha^{k+1}(\muav/2)\by^{k+1} + (1 + A^k\muav/2)\bu^k}{1 + A^{k+1}\muav/2} - \frac{\alpha^{k+1}\nabla F(\by^{k+1})}{1 + A^{k+1}\muav/2}$}
		\vspace{0.1cm}
		\STATE{\label{alg_step:consensus_update_prox}
			$
			\bu^{k+1} = \prox_G^{\gamma_k}\cbraces{\text{Consensus}(\bv^{k+1}, T)}
			$
		}
		\vspace{0.1cm}
		\STATE{$\ds \bx^{k+1} = \frac{\alpha^{k+1} \bu^{k+1} + A^k \bx^k}{A^{k+1}}$}
		\ENDFOR
	\end{algorithmic}
\end{algorithm}
Define $\gamma^k = \frac{\alpha_{k+1}}{1 + A_{k+1}\muav/2},~ \hat\bx^0 = \bx^0,~ \hat\by = \by^0,~ \hat\bu^0 = \bu^0$ and consider a method which trajectory lies in $\cL$:
\begin{align*}
	\hat\by^{k+1} &= \frac{\alpha_{k+1}\hat\bu^k + A_k\hat\bx^k}{A_{k+1}} \\
	\hat\bu^{k+1} &= \prox_G^{\gamma_k} \sbraces{\mu\gamma_k\hat\by^{k+1} + (1 - \mu\gamma_k)\hat\bu^k - \gamma_k\mP\nabla F(\by^{k+1})} \\
	\hat\bx^{k+1} &= \frac{\alpha^{k+1}\hat\bu^{k+1} + A_k\hat\bx^k}{A_{k+1}}.
\end{align*}
Introduce $\mW_\tau^k = \mW^k\ldots\mW^{k-\tau+1}$ for $k\geq\tau - 1$. We have 
\begin{align*}
	\norm{\mW_\tau^k - \mP}_2
	&= \norm{(\mW^k - \mP)\ldots(\mW^{k-\tau+1} - \mP)}_2 \\
	&\leq \norm{\mW^k - \mP}_2 \ldots \norm{\mW^{k-\tau+1} - \mP}_2
	\leq (1 - \chi^{-1})^T.
\end{align*}
Introduce
\begin{align}
	\lambda = (1 - \chi^{-1})^T.
\end{align}

\begin{lemma}\label{lem:beta_recurrence}
	For $k\geq 0$ define $\beta_k = \max\cbraces{\norm{\by_k - \hat\by_k}_2, \norm{\bu_k - \hat\bu_k}_2, \norm{\bx_k - \hat\bx_k}_2}$. We have
	\begin{align*}
		\beta_{k+1}\leq (1 + \lambda)\beta_k + \lambda\gamma_k\norm{\nabla F(\by^{k+1})}.
	\end{align*}
\end{lemma}
\begin{proof}
	First, we have that 
	\begin{align*}
		\norm{\by^{k+1} - \hat\by^{k+1}}_2\leq \frac{\alpha_{k+1}}{A_{k+1}}\norm{\bu^k - \hat\bu^k} + \frac{A_k}{A_{k+1}}\norm{\bx^k - \hat\bx^k}_2\leq \beta_k.
	\end{align*}
	In particular, 
	\begin{align*}
		\norm{\by^{k+1} - \hat\by^{k+1}}_2\leq (1 + \lambda)\beta_k+ \lambda\gamma_k\norm{\nabla F(\by^{k+1})}_2.
	\end{align*}
	After that, consider an update rule for $\bu^{k+1}$. We have
	\begin{align*}
		\|\bu^{k+1} &- \hat\bu^{k+1}\|_2 \\
		&\overset{\circledOne}{\leq} \Big\|\prox_G^{\gamma_k}\sbraces{\mW_\tau^{(k+1)\tau-1}\cbraces{\mu\gamma_k\by^{k+1} + (1 - \mu\gamma_k)\bu^k - \gamma_k\nabla F(\by^{k+1})}} - \\
		&\qquad\prox_G^{\gamma_k} \sbraces{\mu\gamma_k\hat\by^{k+1} + (1 - \mu\gamma_k)\hat\bu^k - \gamma_k\mP\nabla F(\by^{k+1})}\Big\|_2 \\
		&\leq \|\mu\gamma_k(\mW_\tau^{(k+1)\tau-1}\by^{k+1} - \hat\by^{k+1}) + (1 - \mu\gamma_k)(\mW_\tau^{(k+1)\tau-1}\bu^k - \hat\bu^k) - \\ &\qquad-\gamma_k(\mW_\tau^{(k+1)\tau-1}\nabla F(\by^{k+1}) - \mP\nabla F(\by^{k+1}))\|_2 \\
		&\leq \mu\gamma_k\norm{(\mW_\tau^{(k+1)\tau-1} - \mP)(\by^{k+1} - \mP\by^{k+1}) + (\mP\by^{k+1} - \mP\hat\by^{k+1})}_2 \\
		&\qquad+ (1 - \mu\gamma_k)\norm{(\mW_\tau^{(k+1)\tau-1} - \mP)(\bu^k - \mP\bu^k) + (\mP\bu^k - \mP\hat\bu^k)}_2 \\
		&\qquad+ \gamma_k\norm{(\mW_\tau^{(k+1)\tau-1} - \mP)\nabla F(\by^{k+1})}_2 \\
		&\leq \mu\gamma_k(\lambda\beta_k+ \beta_k) + (1 - \mu\gamma_k)(\lambda\beta_k + \beta_k) + \gamma_k\lambda\norm{\nabla F(\by^{k+1})}_2 \\
		&= (1 + \lambda)\beta_k+ \lambda\gamma_k\norm{\nabla F(\by^{k+1})}_2,
	\end{align*}
	where $\circledOne$ holds by non-expansiveness property of prox-operator.
	Finally, for $\bx^{k+1}$ we obtain
	\begin{align*}
		\norm{\bx^{k+1} - \hat\bx^{k+1}}_2
		&\leq \frac{\alpha_{k+1}}{A_{k+1}}\norm{\bu^{k+1} - \hat\bu^{k+1}} + \frac{A_k}{A_{k+1}}\norm{\bx^k - \hat\bx^k}_2\leq \beta_k \\
		&\leq (1 + \lambda)\beta_k+ \lambda\gamma_k\norm{\nabla F(\by^{k+1})}_2.
	\end{align*}
	As a result, we have
	\begin{align*}
		\beta_{k+1} &= \max\cbraces{\norm{\by^{k+1} - \hat\by^{k+1}}_2, \norm{\bu^{k+1} - \hat\bu^{k+1}}_2, \norm{\bx^{k+1} - \hat\bx^{k+1}}_2} \\
		&\leq (1 + \lambda)\beta_k+ \lambda\gamma_k\norm{\nabla F(\by^{k+1})}_2.
	\end{align*}
\end{proof}

We recall a result from \cite{stonyakin2020inexact} revisited in terms of \cite{rogozin2020towards}.
\begin{lemma}\label{lem:coef_bound}
	The following relations hold.
	\begin{align*}
		A^N&\geq \frac{1}{2\Lav}\cbraces{1 + \frac{1}{4}\sqrt{\frac{\muav}{\Lav}}}^{2(N-1)}, \\
		\frac{\sum_{k=0}^{N-1} A^{k+1}}{A^N}&\leq 1 + 2\sqrt{\frac{\Lav}{\muav}}.
	\end{align*}
\end{lemma}

\begin{theorem}
	Let Assumptions~\ref{assum:smoothness} and \ref{assum:strong_convexity} hold. Then Algorithm~\ref{alg:decentralized_agd_prox} requires
	\begin{align*}
		N_{comp} = O\cbraces{\sqrt{\frac{\Lav}{\muav}}\log\cbraces{\frac{1}{\eps}}}
	\end{align*}
	oracle calls per node and
	\begin{align*}
		N_{comm} = O\cbraces{\chi\sqrt{\frac{\Lav}{\muav}}\log\cbraces{\frac{1}{\eps}}}
	\end{align*}
	communication rounds to reach $\eps$-accuracy.
\end{theorem}
\begin{proof}
	Let $\delta = \eta\sum_{k=0}^{N-1}\beta_k^2$. According to Theorem~3.1 in \cite{stonyakin2020inexact} we have
	\begin{align*}
		\norm{\bu^N - \bx^*}_2^2
		&\leq \frac{\norm{\bu^0 - \bx^*}_2^2}{1 + A_N\muav/2} + \frac{4\sum_{k=0}^{N-1} A_{k+1}\delta}{1 + A_N\muav/2} \\
		&\leq \frac{2\norm{\bu^0 - \bx^*}_2^2}{A_N\muav} + \frac{8\sum_{k=0}^{N-1} A_{k+1}\delta}{A_N\muav}, \\
		\norm{\bx^N - \bx^*}_2^2
		&\leq \frac{2}{\muav/2} \cbraces{\frac{\norm{\bu^0 - \bx^*}_2^2}{2A_N} + \frac{2\sum_{k=0}^{N-1}A_{k+1}\delta}{A_N}} \\
		&= \frac{2\norm{\bu^0 - \bx^*}_2^2}{A_N\muav} + \frac{8\sum_{k=0}^{N-1} A_{k+1}\delta}{A_N\muav}.
	\end{align*}
	First, by Lemma~\ref{lem:coef_bound} we have $A^N\geq \frac{1}{2\Lav}\cbraces{1 + \frac{1}{4}\sqrt{\frac{\muav}{\Lav}}}^{2(N-1)}$. To ensure condition $\norm{\bx^N - \bx^*}_2^2\leq\eps/2$ it is sufficient to set $N = O\cbraces{\sqrt{\Lav/\muav}\log(1/\eps)}$.
	
	Using convexity of $\norm{\cdot}_2^2$ we obtain 
	\begin{align*}
		\norm{\by^{k+1} - \bx^*}_2^2
		&= \norm{\frac{\alpha^{k+1}\bu^k}{A^{k+1}} + \frac{A^k\bx^k}{A^{k+1}} - \bx^*}_2^2 \\
		&\leq \frac{\alpha^{k+1}}{A^{k+1}} \norm{\bu^k - \bx^*}_2^2 + \frac{A^k}{A^{k+1}}\norm{\bx^k - \bx^*}_2^2 \\
		&\overset{\circledOne}{\leq} \frac{2\norm{\bu^0 - \bx^*}_2^2}{A_{k+1}\muav} + \frac{8\sum_{t=0}^{k} A_{t+1}\delta}{A_{k+1}\muav} \\
		&\leq \frac{4\Lav}{\muav} \norm{\bu^0 - \bx^*}_2^2 + \frac{8\delta}{\muav}\cbraces{1 + 2\sqrt{\frac{\Lav}{\muav}}}.
	\end{align*}
	For brevity introduce
	\begin{align*}
		a = \frac{4\Lav^3}{\muav},~ b = \frac{8\Lav^2}{\muav}\cbraces{1 + 2\sqrt{\frac{\Lav}{\muav}}},~ R_0^2 = \norm{\bu^0 - \bx^*}_2^2.
	\end{align*}
	Therefore, $\norm{\nabla F(\by^{k+1})}_2\leq \Lav\sqrt{aR_0^2 + b\delta}$. By Lemma~\ref{lem:beta_recurrence} we have
	\begin{align*}
		\beta_{k+1}
		&\leq (1 + \lambda)\beta_k + \lambda\gamma\norm{\nabla F(\by^{k+1})}_2 \\
		&\leq (1 + \lambda)\beta_k + \lambda\gamma\cbraces{\sqrt{aR_0^2 + b\delta} + \norm{\nabla F(\bx^*)}_2}.
	\end{align*}
	Unfolding the recurrence we obtain
	\begin{align*}
		\beta_{k+1}
		&\leq \lambda\gamma\cbraces{\sqrt{aR_0^2 + b\delta} + \norm{\nabla F(\bx^*)}_2}\sum_{t=0}^{k} (1 + \lambda)^t \\
		&\leq \lambda\gamma\cbraces{\sqrt{aR_0^2 + b\delta} + \norm{\nabla F(\bx^*)}_2} \frac{(1 + \lambda)^{k+1} - 1}{(1 + \lambda) - 1} \\
		&= \gamma\cbraces{\sqrt{aR_0^2 + b\delta} + \norm{\nabla F(\bx^*)}_2} \cbraces{(1 + \lambda)^{k+1} - 1}.
	\end{align*}
	Summing over $k$, we obtain
	\begin{align*}
		\delta &= \eta \sum_{k=0}^{N-1} \beta_k^2
		\leq \eta N\beta_{N-1}^2 \\
		&\leq \eta \gamma^2 \cbraces{\sqrt{aR_0^2 + b\delta} + \norm{\nabla F(\bx^*)}_2}^2\cdot N\cbraces{(1 + \lambda)^{N-1} - 1}^2.
	\end{align*}
	Denote $c(N) = N\cbraces{(1 + \lambda)^{N-1} - 1}^2$.
	\begin{align*}
		\frac{\delta}{c(N)}
		&\leq \eta\gamma^2 \cbraces{\sqrt{aR_0^2 + b\delta} + \norm{\nabla F(\bx^*)}_2}^2 \\
		&\leq 2\eta\gamma^2 \cbraces{aR_0^2 + b\delta+ \norm{\nabla F(\bx^*)}_2^2}, \\
		\delta\cbraces{\frac{1}{c(N)} - 2\eta\gamma^2 b}&\leq 2\gamma^2\cbraces{aR_0^2 + \norm{\nabla F(\bx^*)}_2^2}, \\
		\delta &\leq 2\gamma^2\cbraces{aR_0^2 + \norm{\nabla F(\bx^*)}_2^2}\cbraces{\frac{1}{c(N)} - 2\eta\gamma^2 b}^{-1}, \\
		&= \frac{aR_0^2 + \norm{\nabla F(\by^{k+1})}_2^2}{\eta b} \cbraces{\frac{1}{1 - 2\gamma^2 bc(N)} - 1}.
	\end{align*}
	We would like to choose number of consensus iterations $T$ such that $\delta\leq\eps/2$. Setting
	\begin{align*}
		T&\geq \frac{\chi}{2}\log\sbraces{\frac{32N^3}{\eta\muav\Lav\eps}\cbraces{aR_0^2 + \norm{\nabla F(\bx^*)}_2^2}}\\
		&\geq \frac{\chi}{2}\log\sbraces{\frac{32N^3}{\eta\muav\Lav\eps}\cbraces{\frac{4\Lav^3}{\muav}\norm{\bu^0 - \bx^*}_2^2 + \norm{\nabla F(\bx^*)}_2^2}} \\
		&= O\cbraces{\chi\log\cbraces{\frac{N^3}{\eps}}} = O\cbraces{\chi\log\cbraces{\frac{1}{\eps}} + \chi\log\log\cbraces{\frac{1}{\eps}}}
		= O\cbraces{\chi\log\cbraces{\frac{1}{\eps}}}.
	\end{align*}
	we obtain
	\begin{align*}
		\lambda &= (1 - \chi^{-1})^T = \exp\cbraces{T\log(1 - \chi^{-1})}\leq \exp(-T/\chi) \\
		&\leq \frac{\sqrt{\eta\muav\Lav}}{4N}\sqrt{\frac{\eps}{2N}} \cbraces{\frac{4\Lav^3}{\muav}\norm{\bu^0 - \bx^*}_2^2 + \norm{\nabla F(\bx^*)}_2^2}^{-1/2} \\
		&= \frac{1}{4N}\cbraces{\frac{\eta\eps}{2N\gamma^2\cbraces{aR_0^2 + \norm{\nabla F(\bx^*)}_2^2}}}^{1/2}
	\end{align*}
	For any $0\leq x\leq 1$ it holds $1 - \frac{1}{1 + x}\geq 1 - \cbraces{1 - \frac{x}{2}} = \frac{x}{2}$. Assuming that $\eps$ is sufficiently small, we obtain
	\begin{align*}
		\lambda&\leq \frac{1}{N} \sbraces{1 - \cbraces{1+ \cbraces{\frac{\eta\eps}{8N\gamma^2\cbraces{aR_0^2 + \norm{\nabla F(\bx^*)}_2^2}}}^{1/2}}^{-1}}, \\
		1 - N\lambda&\geq \sbraces{1+ \cbraces{\frac{\eta\eps}{8N\gamma^2(aR_0^2 + \norm{\nabla F(\bx^*)}_2^2)}}^{1/2}}^{-1}, \\
		\frac{1}{1 - N\lambda} - 1&\leq \cbraces{\frac{\eta\eps}{8N\gamma^2(aR_0^2 + \norm{\nabla F(\bx^*)}_2^2)}}^{1/2}.
	\end{align*}
	After that, let us note that 
	\begin{align*}
		(1 + \lambda)^{N-1} - 1
		= \sum_{k=1}^{N-1} \binom{N-1}{k} \lambda^k
		\leq \sum_{k=1}^{N-1} N^k\lambda^k
		\leq \frac{N\lambda}{1 - N\lambda}
		= \frac{1}{1 - N\lambda} - 1.
	\end{align*}
	Consequently,
	\begin{align*}
		c(N) &= N\cbraces{(1 + \lambda)^{N-1} - 1}^2\leq \frac{\eta\eps}{8\gamma^2\cbraces{aR_0^2 + \norm{\nabla F(\bx^*)}_2^2}}, \\
		1 - 2\gamma^2 b c(N)&\geq 1 - \frac{\eta b\eps}{4(aR_0^2 + \norm{\nabla F(\bx^*)})}.
	\end{align*}
	We once again use that $\frac{1}{1 + x}\leq 1 - \frac{x}{2}$ for $0\leq x\leq 1$.
	\begin{align*}
		1 - 2\gamma^2bc(N)&\geq \cbraces{1+ \frac{\eta b\eps}{2(aR_0^2 + \norm{\nabla F(\bx^*)}_2^2)}}^{-1} \\
		\frac{1}{1 - 2\gamma^2bc(N)}&\leq 1 + \frac{\eta b\eps}{2(aR_0^2 + \norm{\nabla F(\bx^*)}_2^2)}
	\end{align*}
	Finally, we get
	\begin{align*}
		\delta\leq \frac{aR_0^2 + \norm{\nabla F(\bx^*)}_2^2}{\eta b}\cbraces{\frac{1}{1 - 2\gamma^2bc(N)} - 1}\leq \frac{\eps}{2}.
	\end{align*}
	Summing up, to reach $\eps$-accuracy Algorithm~\ref{alg:decentralized_agd_prox} requires
	\begin{align*}
		N_{comm} &= NT = O\cbraces{\chi\sqrt{\frac{\Lav}{\muav}}\log\cbraces{\frac{1}{\eps}}}, \\
		N_{comp} &= N = O\cbraces{\sqrt{\frac{\Lav}{\muav}}\log\cbraces{\frac{1}{\eps}}}.
	\end{align*}
\end{proof}

\section{Conclusion}

We propose an consensus subroutine approach that works for composite optimization. The algorithm is an accelerated proximal method that performs several communication rounds after each proximal step. The novelty of the approach is that we do not bound the gradient norm. We apply a novel proof technique and thus get an algorithm that recovers lower communication complexity bounds up to a logarithmic factor.

\bibliographystyle{abbrv}
\bibliography{references}

\begin{thebibliography}{10}

\bibitem{bazerque2009distributed}
J.~A. Bazerque and G.~B. Giannakis.
\newblock Distributed spectrum sensing for cognitive radio networks by
  exploiting sparsity.
\newblock {\em IEEE Transactions on Signal Processing}, 58(3):1847--1862, 2009.

\bibitem{beznosikov2021near}
A.~Beznosikov, A.~Rogozin, D.~Kovalev, and A.~Gasnikov.
\newblock Near-optimal decentralized algorithms for saddle point problems over
  time-varying networks.
\newblock In {\em International Conference on Optimization and Applications},
  pages 246--257. Springer, 2021.

\bibitem{beznosikov2021distributed_2}
A.~Beznosikov, V.~Samokhin, and A.~Gasnikov.
\newblock Distributed saddle-point problems: Lower bounds, optimal algorithms
  and federated gans.
\newblock {\em arXiv preprint arXiv:2010.13112}, 2021.

\bibitem{cai2014average}
K.~Cai and H.~Ishii.
\newblock Average consensus on arbitrary strongly connected digraphs with
  time-varying topologies.
\newblock {\em IEEE Transactions on Automatic Control}, 59(4):1066--1071, 2014.

\bibitem{dvinskikh2019decentralized}
D.~Dvinskikh and A.~Gasnikov.
\newblock Decentralized and parallel primal and dual accelerated methods for
  stochastic convex programming problems.
\newblock {\em Journal of Inverse and Ill-posed Problems}, 29(3):385--405,
  2021.

\bibitem{forero2010consensus}
P.~A. Forero, A.~Cano, and G.~B. Giannakis.
\newblock Consensus-based distributed support vector machines.
\newblock {\em Journal of Machine Learning Research}, 11(5), 2010.

\bibitem{gan2012optimal}
L.~Gan, U.~Topcu, and S.~H. Low.
\newblock Optimal decentralized protocol for electric vehicle charging.
\newblock {\em IEEE Transactions on Power Systems}, 28(2):940--951, 2012.

\bibitem{jadbabaie2003coordination}
A.~Jadbabaie, J.~Lin, and A.~S. Morse.
\newblock Coordination of groups of mobile autonomous agents using nearest
  neighbor rules.
\newblock {\em IEEE Transactions on automatic control}, 48(6):988--1001, 2003.

\bibitem{jakovetic2014fast}
D.~Jakoveti{\'c}, J.~Xavier, and J.~M. Moura.
\newblock Fast distributed gradient methods.
\newblock {\em IEEE Transactions on Automatic Control}, 59(5):1131--1146, 2014.

\bibitem{kovalev2020optimal}
D.~Kovalev, A.~Salim, and P.~Richt{\'a}rik.
\newblock Optimal and practical algorithms for smooth and strongly convex
  decentralized optimization.
\newblock {\em Advances in Neural Information Processing Systems}, 33, 2020.

\bibitem{nedic2020distributed}
A.~Nedic.
\newblock Distributed gradient methods for convex machine learning problems in
  networks: Distributed optimization.
\newblock {\em IEEE Signal Processing Magazine}, 37(3):92--101, 2020.

\bibitem{nedic2017fast}
A.~Nedi{\'c}, A.~Olshevsky, and C.~A. Uribe.
\newblock Fast convergence rates for distributed non-bayesian learning.
\newblock {\em IEEE Transactions on Automatic Control}, 62(11):5538--5553,
  2017.

\bibitem{olshevsky2010efficient}
A.~Olshevsky.
\newblock Efficient information aggregation strategies for distributed control
  and signal processing.
\newblock {\em arXiv preprint arXiv:1009.6036}, 2010.

\bibitem{olshevsky2014linear}
A.~Olshevsky.
\newblock Linear time average consensus on fixed graphs and implications for
  decentralized optimization and multi-agent control.
\newblock {\em arXiv preprint arXiv:1411.4186}, 2014.

\bibitem{rabbat2004distributed}
M.~Rabbat and R.~Nowak.
\newblock Distributed optimization in sensor networks.
\newblock In {\em Proceedings of the 3rd international symposium on Information
  processing in sensor networks}, pages 20--27, 2004.

\bibitem{ram2009distributed}
S.~S. Ram, V.~V. Veeravalli, and A.~Nedic.
\newblock Distributed non-autonomous power control through distributed convex
  optimization.
\newblock In {\em IEEE INFOCOM 2009}, pages 3001--3005. IEEE, 2009.

\bibitem{ren2006consensus}
W.~Ren.
\newblock Consensus based formation control strategies for multi-vehicle
  systems.
\newblock In {\em 2006 American Control Conference}, pages 6--pp. IEEE, 2006.

\bibitem{ren2008distributed}
W.~Ren and R.~W. Beard.
\newblock {\em Distributed consensus in multi-vehicle cooperative control},
  volume~27.
\newblock Springer, 2008.

\bibitem{rogozin2021accelerated}
A.~Rogozin, M.~Bochko, P.~Dvurechensky, A.~Gasnikov, and V.~Lukoshkin.
\newblock An accelerated method for decentralized distributed stochastic
  optimization over time-varying graphs.
\newblock {\em Conference on decision and control}, 2021.

\bibitem{rogozin2022decentralized_survey}
A.~Rogozin, A.~Gasnikov, A.~Beznosikov, and D.~Kovalev.
\newblock Decentralized optimization over time-varying graphs: a survey.
\newblock {\em arXiv preprint arXiv:2210.09719}, 2022.

\bibitem{rogozin2021towards}
A.~Rogozin, V.~Lukoshkin, A.~Gasnikov, D.~Kovalev, and E.~Shulgin.
\newblock Towards accelerated rates for distributed optimization over
  time-varying networks.
\newblock In {\em International Conference on Optimization and Applications},
  pages 258--272. Springer, 2021.

\bibitem{scaman2017optimal}
K.~Scaman, F.~Bach, S.~Bubeck, Y.~T. Lee, and L.~Massouli{\'e}.
\newblock Optimal algorithms for smooth and strongly convex distributed
  optimization in networks.
\newblock In {\em Proceedings of the 34th International Conference on Machine
  Learning-Volume 70}, pages 3027--3036. JMLR. org, 2017.

\bibitem{scaman2018optimal}
K.~Scaman, F.~Bach, S.~Bubeck, L.~Massouli{\'e}, and Y.~T. Lee.
\newblock Optimal algorithms for non-smooth distributed optimization in
  networks.
\newblock In {\em Advances in Neural Information Processing Systems}, pages
  2740--2749, 2018.

\bibitem{stonyakin2020inexact}
F.~Stonyakin, A.~Tyurin, A.~Gasnikov, P.~Dvurechensky, A.~Agafonov,
  D.~Dvinskikh, M.~Alkousa, D.~Pasechnyuk, S.~Artamonov, and V.~Piskunova.
\newblock Inexact model: A framework for optimization and variational
  inequalities.
\newblock {\em Optimization Methods and Software}, pages 1--47, 2021.

\bibitem{xiao2007distributed}
L.~Xiao, S.~Boyd, and S.-J. Kim.
\newblock Distributed average consensus with least-mean-square deviation.
\newblock {\em Journal of parallel and distributed computing}, 67(1):33--46,
  2007.

\bibitem{ye2020multi}
H.~Ye, L.~Luo, Z.~Zhou, and T.~Zhang.
\newblock Multi-consensus decentralized accelerated gradient descent.
\newblock {\em arXiv preprint arXiv:2005.00797}, 2020.

\end{thebibliography}

\end{document}